\documentclass{article}

\usepackage{booktabs}       
\usepackage{amsfonts}       
\usepackage{nicefrac}       
\usepackage{microtype}      

\usepackage{fncylab,enumerate,wrapfig,placeins}
\usepackage{bbm}
\usepackage{graphicx}
\usepackage{enumitem}  
\usepackage{amsmath,amsthm,amsfonts,amssymb,xcolor}  
\usepackage{pifont}

\usepackage{geometry}
\usepackage{leftindex}

\pdfoutput=1  

\usepackage{hyperref}       
\usepackage{url}            

\usepackage{graphicx}  
\usepackage{relsize} 
\usepackage{accents}  

\usepackage{caption}
\usepackage{textgreek,upgreek,bm}

\usepackage{titlesec}

\newtheorem{theorem}{Theorem}[section]
\newtheorem{lemma}[theorem]{Lemma}

\newtheorem{corollary}[theorem]{Corollary}

\newlength{\noteWidth}
\long\def\notes#1{\ifinner
	{\footnotesize #1}
	\else 
	\marginpar{\parbox[t]{\noteWidth}{\raggedright\tiny#1}}  
	\fi\typeout{#1}}
	
\usepackage{cleveref}

\Crefname{corollary}{Corollary}{Corollaries}
\Crefname{eqnarray}{eq.}{eqs.}
\Crefname{equation}{eq.}{eqs.}

\Crefname{figure}{Fig.}{Figs.}
\Crefname{tabular}{Tab.}{Tabs.}
\Crefname{table}{Tab.}{Tabs.}
\Crefname{proposition}{Prop.}{Propositions}
\Crefname{theorem}{Thm.}{Thms.}
\Crefname{definition}{Def.}{Defs.} 
\Crefname{section}{Section}{Sections}
\Crefname{lemma}{Lemma}{Lemmas}
\Crefname{assumption}{Assumption}{Assumptions}

\def\haUppsi{\widehat{\Uppsi}}

\def\haXi{\widehat{\Xi}}
\def\haclL{\widehat{\clL}}

\def\urls#1{{\footnotesize\url{#1}}}

 \def\uPO{u^{  \text {\tiny \sf P}}}
 \def\PhiPO{\Phi^{  \text {\tiny \sf P}}}

  \def\zPO{z^{  \text {\tiny \sf P}}}

  \def\lambdaPO{\lambda^{  \text {\tiny \sf P}}}

\def\clD{\mathcal{D}}

\def\mindex#1{\index{#1}}

\DeclareFontFamily{U}{mathx}{\hyphenchar\font45}
\DeclareFontShape{U}{mathx}{m}{n}{<-> mathx10}{}
\DeclareSymbolFont{mathx}{U}{mathx}{m}{n}
\DeclareMathAccent{\widebar}{0}{mathx}{"73}

\def\Obj{\Upgamma}  
\def\barObj{\bar{\Obj}}

\newcommand{\bbblot}{\raise1pt\hbox{\vrule height .4ex width .4ex depth .05ex}}

\long\def\defbox#1{\framebox[.9\hsize][c]{\parbox{.85\hsize}{%
\parindent=0pt
\baselineskip=12pt plus .1pt      
\parskip=6pt plus 1.5pt minus 1pt 
 #1}}}

\long\def\beginbox#1\endbox{\subsection*{}%
\hbox{\hspace{.05\hsize}\defbox{\medskip#1\bigskip}}%
\subsection*{}}

\def\endbox{}

 \def\archival#1{} 

\def\FRAC#1#2#3{\genfrac{}{}{}{#1}{#2}{#3}}

\def\ddtp{{\mathchoice{\FRAC{1}{d^{\hbox to 2pt{\rm\tiny +\hss}}}{dt}}%
{\FRAC{1}{d^{\hbox to 2pt{\rm\tiny +\hss}}}{dt}}%
{\FRAC{3}{d^{\hbox to 2pt{\rm\tiny +\hss}}}{dt}}%
{\FRAC{3}{d^{\hbox to 2pt{\rm\tiny +\hss}}}{dt}}}}

\def\ddyp{{\mathchoice{\FRAC{1}{d^{\hbox to 2pt{\rm\tiny +\hss}}}{dy}}%
{\FRAC{1}{d^{\hbox to 2pt{\rm\tiny +\hss}}}{dy}}%
{\FRAC{3}{d^{\hbox to 2pt{\rm\tiny +\hss}}}{dy}}%
{\FRAC{3}{d^{\hbox to 2pt{\rm\tiny +\hss}}}{dy}}}}

\def\limsup{\mathop{\rm lim{\,}sup}}

\def\bfr{{\bf r}}

\def\bfmath#1{{\mathchoice{\mbox{\boldmath$#1$}}%
{\mbox{\boldmath$#1$}}%
{\mbox{\boldmath$\scriptstyle#1$}}%
{\mbox{\boldmath$\scriptscriptstyle#1$}}}}

\def\bfmY{\bfmath{Y}}

\def\bfmhhaY{\bfmath{\hhaY}} 
\def\bfmhhaY{\hbox to 0pt{$\widehat{\bfmY}$\hss}\widehat{\phantom{\raise 1.25pt\hbox{$\bfmY$}}}}

\def\hay{{\hat y}}

\def\tilz{\tilde z}

\def\clB{{\cal B}}

\def\clE{{\cal E}}
\def\clF{{\cal F}}
\def\clG{{\cal G}}
\def\clH{{\cal H}}

\def\clL{{\cal L}}
\def\clM{{\cal M}}

\def\clQ{{\cal Q}}
\def\clR{{\cal R}}

\def\clU{{\cal U}}

\def\clZ{{\cal Z}}

\def\clL{{\cal L}}

\def\barclL{\bar{\cal L}}

\def\eqdef{\mathbin{:=}}

\def\Expect{{\sf E}}

\def\Proj{\hbox{\sf Proj}}

 \def\epsy{\varepsilon}

\def\varble{\,\cdot\,}

\def\formtmp#1#2{{\vskip12pt\noindent\fboxsep=0pt\colorbox{#1}{\vbox{\vskip3pt\hbox to \textwidth{\hskip3pt\vbox{\raggedright\noindent\textbf{#2\vphantom{Qy}}}\hfill}\vspace*{3pt}}}\par\vskip2pt%
\noindent\kern0pt}}

\def\barb{{\overline {b}}}

\def\barP{{\bar{P}}}

\def\barpsi{{\bar{\psi}}}

{\end{list}}

\def\ass(#1:#2){(#1\ref{#1:#2})}

\def\ritem#1{
\item[{\sf \ass(\current_model:#1)}]
}

\newenvironment{recall-ass}[1]{%
\begin{description}
\def\current_model{#1}}{
\end{description}
}

\def\sq{\hbox{\rlap{$\sqcap$}$\sqcup$}}
\def\qed{\ifmmode\sq\else{\unskip\nobreak\hfil
\penalty50\hskip1em\null\nobreak\hfil\sq
\parfillskip=0pt\finalhyphendemerits=0\endgraf}\fi}

\newcommand{\blot}{\vrule height 1.1ex width .9ex depth -.1ex }
\def\qedb{\ifmmode\blot\else{\vspace{-.2cm}\unskip\nobreak\hfil
\penalty50\hskip1em\null\nobreak\hfil\blot
\parfillskip=0pt\finalhyphendemerits=0\endgraf}\fi}

\newcounter{rmnum}

\newcounter{anum}

\newcommand{\field}[1]{\mathbb{#1}}

\def\Re{\field{R}}

\def\Expect{{\sf E}}

\def\transpose{{\intercal}}

\def\epsy{\varepsilon}
\def\varble{\,\cdot\,}

\def\haY{\widehat{Y}}

\def\hhaY{\hbox to 0pt{$\haY$\hss}\widehat{\phantom{\raise 1.25pt\hbox{Y}}}}
 
\def\haXi{\widehat{\Xi}}

\def\haY{\widehat Y}

\newlength{\dhatheight}

\def\tilUppsi{\widetilde{\Uppsi}}

\def\barUppsi{\widebar{\Uppsi}}

\makeatletter
\newcommand\gobblepars{%
	\@ifnextchar\par%
	{\expandafter\gobblepars\@gobble}%
	{}}
\makeatother

\def\whamit#1{\smallbreak\pagebreak[3]%
	\noindent\textit{#1}\ \ \gobblepars}

\def\wham#1{\smallbreak\pagebreak[3]%
	\noindent\textbf{#1}\ \ \gobblepars}
	
\def\whamrm#1{\smallbreak\pagebreak[3]%
	\noindent{{\upshape\rm#1}}\ \ \gobblepars}

\title{Online Feedback Optimization for 
Constrained Stochastic Problems 
\\
with Decision-Dependent Distributions
}

\author{Caio Kalil Lauand and Emiliano Dall'Anese
\thanks{C. K. Lauand is with the Division of Systems Engineering, Boston University, Boston, MA, USA. E. Dall'Anese is with the Department of Electrical and Computer Engineering and the Division of Systems Engineering, Boston University, Boston, MA, USA. Emails: {\tt\small cklauand@bu.edu},  {\tt\small  edallane@bu.edu}. This work was supported by the Division of Systems Engineering at Boston University and by the National Science Foundation Award 2504084.
 }%
 }%

\begin{document}

\maketitle

\begin{abstract}
Online feedback optimization (OFO) leverages real-time output 
measurements to  optimize the operation of networked systems without requiring full knowledge of system 
dynamics or disturbances. 
We develop an OFO approach for constrained stochastic optimization problems in which the distribution of the system's random parameters shifts in response to the control actions. We propose a projected primal-dual algorithm where the true dual constraint sets are replaced by surrogate sets. Our main result is an upper bound on the mean-square tracking error, which decomposes into four interpretable terms reflecting: (i) the stochasticity of the problem, (ii) output measurement errors, (iii) time-variability of the problem, and (iv) the mismatch between surrogate and true dual constraint sets. The theory is illustrated in a numerical experiment for power grids with price-responsive assets.
\end{abstract}

\bigskip

\noindent
\textbf{This manuscript is an extended version of an article accepted for publication in IEEE Control Systems Letters (L-CSS).}

\clearpage
\tableofcontents

\clearpage
\section{Introduction}
	\label{s:Intro}
Online feedback optimization (OFO) provides a systematic approach to embed optimization algorithms within feedback controllers~\cite{brunner2012feedback,colombino2019online,berdal19,hauswirth2024optimization}. By leveraging real-time measurements, OFO methods steer system outputs toward solutions of optimization problems without requiring explicit knowledge of system dynamics or disturbances. In this paper, we study OFO in stochastic systems, modeled as algebraic maps, where the distribution of the uncertain parameters depends on the control input~\cite{druxia23,perzrnmenhar20,woobiadal23}. This setting arises in several engineering applications but poses significant analytical challenges.
   
Consider then the setting of a linear (or linearized) networked system described by
\begin{equation}
y_n = h^{(n)}(u_n,\Phi_n) \eqdef F u_n + G \Phi_n + D r_n ,
\label{e:model_lin}
\end{equation}
where $n$ denotes the discrete-time index, $y_n \in \Re^m$ represents the system output, $u_n \in \Re^d$ the control input, and $r_n \in \Re^p$ an exogenous disturbance process. The superscript on $h$ indicates that the system map may vary with time; in particular, this time variation may arise from the disturbance sequence $\mathbf r := \{r_n\}$. For each time $n$, the random variable $\Phi_{n} \in \Re^q$  has a distribution that depends on the current input $u_n$.

 We are interested in solving the following time-varying constrained optimization problem:
\begin{subequations}
\begin{align}
&\min_{u \in  \clU^{(n)}}\Expect_{\Phi \sim \clD(u)}[\Obj^{(n)}(u,\Phi)]
\\
& \text{subject to: }
\Expect_{\Phi \sim \clD(u)}[\Xi^{(n)}_i(u,\Phi)] \leq 0
\, , \quad   i = 1, \cdots , M \, ,
\end{align}
\label{e:Opt_prob}
\end{subequations}
\noindent where $\clD(u)$ denotes the probability distribution of the random variable $\Phi$ induced by the input $u$. The problem is time-varying in the sense that both the cost and the constraints may change at each time index $n$. The set $\clU^{(n)}$ represents the system's input constraints (e.g., physical or engineering constraints) at time $n$, $\Obj^{(n)}$ represents the system's performance objective, and the functions $\Xi_i^{(n)}(u,\Phi)$ represent time-varying constraints on the system's output.


The objective function $\Obj^{(n)}$ quantifies performance in terms of outputs and inputs of the system via the additive model $\Obj^{(n)}(u,\Phi) =  \Obj_y^{(n)}(u,\Phi) +  \Obj_u^{(n)}(u,\Phi)$. Specifically, $\Obj_y^{(n)}(u,\Phi) = (g^{(n)}_y \circ h^{(n)})(u,\Phi)$, where $g^{(n)}_y(y)$ measures performance in terms of the system's output. Performance in terms of the system's controllable inputs is measured through $\Obj_u^{(n)}(u,\Phi)= g_u^{(n)}(u,\Phi)$. Throughout the paper we consider the convex optimization case, where $g_y$ is  convex with Lipschitz gradient. Moreover, $g_u$ is assumed  convex in its first argument with Lipschitz gradient. Additionally, $\Xi_i^{(n)}(u,\Phi) = ( g^{i,(n)}_\Xi \circ h^{(n)})(u,\Phi)$,
with $g^{i,(n)}_\Xi(y)$ a convex function with Lipschitz gradient for each $i \in \{1,\ldots,M\}$ and each $n$.

A classical approach to solving \eqref{e:Opt_prob} is to estimate the sequence $\{z^*_n\} \eqdef \{ (u^*_n, \lambda^*_n)^\transpose \}$ of saddle points of the expected regularized Lagrangian:

\begin{equation}
\max_{\lambda \in \Lambda^{(n)}} \min_{u\in  \clU^{(n)}} 
\Expect_{\Phi \sim \clD(u)}[\clL^{(n)}(u,\lambda,\Phi)] \,,
\label{e:minmax_prob}
\end{equation}
in which $\clL^{(n)}(u,\lambda,\Phi)  = \Obj^{(n)}(u,\Phi) + \lambda^\transpose \Xi^{(n)}(u,\Phi) + \tfrac{1}{2}\mu \| u \|^2
 - \tfrac{1}{2}\eta \| \lambda \|^2$,
for constants $\mu,\eta>0$, and $\{ \Lambda^{(n)} \}$ denotes the sequence of dual constraint sets. Here we denote $\Xi^{(n)}(u,\Phi)=
 \begin{bmatrix}
     \Xi_1^{(n)}(u,\Phi)
& \cdots &
      \Xi_M^{(n)}(u,\Phi)
\end{bmatrix}^\transpose$.

Obtaining estimates for the saddle points of the above 
problem may be infeasible due to several challenges:  
(i) knowledge about the disturbance process $\bfr$ or the 
exact system model \eqref{e:model_lin} is not 
available in typical engineering applications;  
(ii) for arbitrary distributional maps $\clD$, the expected 
Lagrangian is not guaranteed to be convex--concave 
\cite{woobiadal23}, and even under structural assumptions 
on $\clD$ that restore convexity, the map itself is 
typically unknown in practice; and  
(iii) the optimal dual variables depend on the distributional 
map, making it infeasible to determine a priori the 
sequence of dual constraint sets that contain them.

Solutions for tackling issue (i) have been previously studied within the OFO literature mostly in deterministic settings \cite{bercomchewan23,berdal19,nonmul21}, with the exception of \cite{lauber25} which considered a stochastic plant with i.i.d.\ noise. A common OFO approach to estimating $\{z^*_n\}$ in such deterministic settings is the approximate projected primal-dual scheme in \cite{berdal19}:
\begin{equation}
z_{n+1} = \Proj_{\clU^{(n)} \times \Lambda^{(n)}}
		\{ z_n 
		- \alpha \nabla\haUppsi^{(n)}  \}
        ,
\end{equation}
where $\alpha>0$ is the step-size and $\nabla\haUppsi^{(n)}$ is an approximation of $\nabla \Uppsi^{(n)}  (z, \Phi) =
\begin{bmatrix}
\nabla_u \clL^{(n)}(z,\Phi) &
-\nabla_\lambda \clL^{(n)}(z,\Phi)
\end{bmatrix}^\transpose$ obtained by replacing the true output 
$y_n = h^{(n)}(u_n, \Phi_n)$ with real-time measurements modeled as 
$\hay_n = \zeta(y_n,\epsy^\bullet_n)$, where $\zeta$ denotes a measurement model and $\{\epsy^\bullet_n\}$ an exogenous random disturbance sequence.


As argued in \cite{lauber25}, the projected primal–dual method is expected to succeed even in simple stochastic settings such as the i.i.d.\ setting considered therein, given the extensive literature supporting stochastic approximation algorithms such as stochastic gradient descent (SGD). Unfortunately, this optimistic picture no longer holds in the presence of decision-dependent stochasticity, in view of challenges (ii) and (iii).

To address the challenge related to the distributional map, which is unknown, we follow a common approach employed in stochastic optimization with decision-dependent distributions 
\cite{druxia23,perzrnmenhar20,woobiadal23} and focus on estimating the performatively stable points of $\{\clL^{(n)}\}$. These correspond to decisions that are optimal with respect to the distribution they induce. Formally, we seek the sequence $\{\zPO_n\} \eqdef \{ (\uPO_n, \lambdaPO_n)^\transpose\}$ satisfying
\[
(\uPO_n, \lambdaPO_n) = \arg \max_{\lambda \in \Lambda^{(n)}} \min_{u \in  \clU^{(n)}} 
\Expect_{\PhiPO_{n} \sim \clD(\uPO_n)}[\clL^{(n)}(u,\lambda,\PhiPO_{n})] \, .
\]
The sequence $\{\zPO_n\}$ is generally distinct from the sequence 
$\{z_n^*\}$ of performative optima. Indeed, if 
$\barObj^{(n)}(u;u'):=\Expect_{\Phi\sim\mathcal D(u')}
[\Obj^{(n)}(u,\Phi)]$, then $u_n^*$ minimizes $\barObj^{(n)}(u;u)$, 
whereas $\uPO_n$ satisfies the fixed-point condition 
$\uPO_n\in\arg\min_u \barObj^{(n)}(u;\uPO_n)$. These notions coincide 
only under additional structural conditions.

Moreover, since the dual constraint sets are not known in our setting, we propose the implementation of the following scheme for estimating $\{\zPO_n\}$:
\begin{equation}
z_{n+1} = \Proj_{\clU^{(n)} \times \clH^{(n)}}
\{ z_n - \alpha \nabla\haUppsi^{(n)}  \}
\label{e:PSGD_noreg_online}
\end{equation}
in which $\{\clH^{(n)}\}$ is a surrogate sequence of constraint sets which might not necessarily contain the performatively stable dual variables.

\wham{Contributions:}
This paper extends the OFO framework to constrained stochastic optimization problems with decision-dependent distributions. We simultaneously address (A) decision-dependent stochasticity, (B) measurement-based gradient approximations, and (C) unknown dual constraint sets. We analyze the projected primal–dual algorithm \eqref{e:PSGD_noreg_online} with surrogate dual constraint sets and establish the following steady-state bound on the mean-square error relative to the sequence of performatively stable saddle points:
\begin{equation}
\limsup_{N \to \infty} \Expect[\| z_N - \zPO_N \|^2]
\leq 
\alpha  \rho_a + \rho_b + \frac{1}{\alpha}\Big[\rho_c + b_\circ  \frac{\rho_c}{\alpha} +\rho_d \Big] + 
  \frac{b_\circ }{\alpha^{3/2}}  \sqrt{\rho_c\Big[ \alpha \rho_a + \rho_b+ \frac{\rho_d}{\alpha} \Big]} .  
  \label{e:MSE_lim}
\end{equation}
Here, $b_\circ$ is a constant and the error terms have the following interpretation: $\rho_a$ captures the error related to the stochasticity of the problem; $\rho_b$ corresponds to the error resulting from approximating $\{\nabla \Uppsi^{(n)}\}$ using measurements $\{ \hay_n  \}$; $\rho_c$ captures the error due to the time variability of $\clL$; and $\rho_d$ quantifies the error associated with the surrogate constraint sets $\{ \clH^{(n)}  \}$.

In particular, if $\Phi_n$ is a deterministic function of $u_n$ for all $n$, then $\rho_a = 0$ and \eqref{e:MSE_lim} reduces to a refinement of \cite[Thm.~4]{berdal19} with the inclusion of an error term associated with the use of surrogate dual sets. Moreover, $\rho_d = 0$ if and only if $\Lambda^{(n)} \subseteq \clH^{(n)}$, that is, when the surrogate dual set contains the performatively stable dual variables. If, on the other hand, $\rho_b = \rho_c = \rho_d = 0$, corresponding to a static optimization problem with no measurement errors and no surrogate dual sets, then \eqref{e:MSE_lim} becomes $O(\alpha)$, consistent with rates from the SGD literature.

\wham{Relevant Literature:} \emph{Online Feedback Optimization.} Although OFO methods 
are broadly applicable to networked systems, most 
advancements in the field have been motivated by applications 
to power systems 
\cite{berdal19,lauber25,hauswirth2024optimization}. We refer 
the reader to \cite{hauswirth2024optimization} for a broader 
survey of OFO. As mentioned after \eqref{e:MSE_lim}, the 
work \cite{berdal19} incorporates primal-dual algorithms 
into the OFO framework in a purely deterministic setting with 
known dual constraint sets. The work \cite{lauber25} extends 
OFO to a stochastic setting and obtains a decomposable MSE 
bound similar to \eqref{e:MSE_lim}, but considers only 
unconstrained problems without decision-dependent 
distributions.

\emph{Optimization with Decision-Dependent Distributions.} 
Within the literature on stochastic optimization with 
decision-dependent distributions (or performative 
prediction), our work is most closely related to 
\cite{druxia23,woobiadal23,NEURIPS2020_33e75ff0,he2025decision}, 
all of which consider static problems with no measurement 
noise. In \cite{woobiadal23}, stochastic 
saddle point problems with decision-dependent distributions are studied with known dual constraint sets, proving existence of 
performatively stable points and obtaining step-size 
independent MSE bounds. The bounds in 
\cite{NEURIPS2020_33e75ff0} are step-size-dependent and 
consistent with the SGD literature, 
but are restricted to unconstrained problems. Convergence guarantees for proximal and clipped gradient  methods are 
obtained in \cite{druxia23}, also in the unconstrained 
setting. The reference \cite{he2025decision} tackles the 
non-convex setting and models distribution shifts as a 
dynamical process, but does not consider constraints. Finally, an approach to online convex optimization with stochastic constraints is studied in \cite{yuneewei17}, but this paper only considers decision-independent stochasticity.

\section{Main Results}
\label{s:main}

\subsection{Preliminaries}
\label{s:assump}

\wham{Notation:}
We use $\|\varble\|$ to denote the Euclidean norm for vectors and the induced operator norm for matrices. For a random variable $X$ (vector- or matrix-valued), we denote its $L_p$ moment by $\|X\|_p^p \eqdef \Expect[\|X\|^p]$,
and the conditional moment with respect to the filtration $\clF_n$ by $\|X\|_{p,n}^p \eqdef \Expect[\|X\|^p \mid \clF_n]$. Here the filtration is generated by the history of the algorithm \eqref{e:PSGD_noreg_online}: $\clF_n \eqdef \{ z_0, \Phi_0, \Phi_1, \dots, \Phi_{n-1}, \epsy^\bullet_0, \epsy^\bullet_1, \dots, \epsy^\bullet_{n-1} \}$. For brevity, we write $z =(u,\lambda)^\transpose$, $\barclL^{(n)}(z;u') = \Expect_{\Phi' \sim \clD(u')}[\clL^{(n)}(u,\lambda,\Phi')]$, $\nabla \barUppsi^{(n)}(z; u') = \Expect_{\Phi' \sim \clD(u')} [\nabla\Uppsi^{(n)}(z,\Phi')]$, and $\Phi_n, \PhiPO_n$ to denote random variables distributed according to $\clD(u_n)$ and $\clD(\uPO_n)$, respectively. 

We start by listing the main assumptions.

\wham{(A1)}
The function $g^{(n)}_y$ is convex and continuously differentiable, while $g^{(n)}_u$ is convex and continuously differentiable in its first argument. Moreover, $g^{i,(n)}_\Xi$ is convex for each $n$ and $1 \le i \le M$. Their gradients are Lipschitz continuous: there exists $L_g,b^\bullet < \infty$ such that for all $n$, $i$, and all $u,u' \in \Re^d$, $\Phi,\Phi' \in \Re^q$, $y,y' \in \Re^m$,

\[
\begin{aligned}
\| \nabla g^{(n)}_y(y) - \nabla g^{(n)}_y(y') \|
&\le L_g \|y-y'\|, \\
\| \nabla_u g^{(n)}_u(u,\Phi) - \nabla_u g^{(n)}_u(u',\Phi') \|
&\le L_g \big(\|u-u'\| + \|\Phi-\Phi'\|\big), \\
\| \nabla g^{i,(n)}_\Xi(y) - \nabla g^{i,(n)}_\Xi(y') \|
&\le L_g \|y-y'\| \, ,
\\
 \max\{ |g^{i,(n)}_\Xi(0) |,  \|\nabla g^{i,(n)}_\Xi(0) \|,  \|  \nabla & g^{(n)}_y (0) \|, \|  \nabla_u g^{(n)}_u(0,0) \| \} \leq b^\bullet \, ,
\end{aligned}
\]
and $| g^{(n)}_u(0,\Phi)| \leq b^\bullet(1+\|\Phi\|^2)$.


\wham{(A2)}
For each $n$, the set $\clU^{(n)} \subset \Re^d$ is convex and compact. Define $b_\clU^{(n)} \eqdef \sup_{u \in \clU^{(n)}} \|u\|$. The surrogate dual set is
$
\clH^{(n)} \eqdef \left\{ \lambda \in \Re_+^M : \sum_{i=1}^M \lambda_i \le b_\clH^{(n)} \right\}
$
for some $b_\clH^{(n)} > 0$. The sequences $\{\clU^{(n)}\}$ and $\{\clH^{(n)}\}$ are uniformly bounded: $\barb_\clU \eqdef \sup_{n\ge1} b_\clU^{(n)}$, $\barb_\clH \eqdef \sup_{n\ge1} b_\clH^{(n)}$,
and we define $b^{\clU,\clH} \eqdef \max\{\barb_\clU,\barb_\clH\}$.

\wham{(A3)} For any $u \in \clU^{(n)}$ and a stationary distribution $\uppi$ independent of $u$, $\Phi \overset{d}{=} \nu(u,\gamma)$,
 where $\gamma \sim \uppi$. Moreover, there exist constants $\sigma_\Delta, L_\nu < \infty$ such that for any $u,u' \in \clU^{(n)}$, 
 \begin{equation}
 \Expect_{\gamma \sim \uppi}[\| \nu(u,\gamma) - \nu(u',\gamma)\|^4] 
 \leq L_\nu^4 \| u-u'\|^4 ,
 \label{e:Phi_Lip}
 \end{equation} 
 $\Expect_{\gamma \sim \uppi}[\|\nu(u,\gamma)\|^4] \le \sigma^4_\Delta$,  and $\sup_n \|r_n\| \le \sigma_\Delta$.
\wham{(A4)}
There exists $\epsy_m < \infty$ such that $\Expect[\|\hat y_n - y_n\|^4 \mid \clF_n] \le \epsy_m^4$ for all $n \ge 0$.

\wham{(A5)}
Slater’s constraint qualification holds for the optimization problem \eqref{e:Opt_prob} at each time instant $n$.

Assumptions (A1), (A2), and (A5) are standard in constrained stochastic optimization, while (A4) is common in the OFO framework \cite{berdal19,bercomchewan23}.  Moreover, Assumption (A3)  implies that the distributional map $\clD$ is Lipschitz in the Wasserstein-4, 2, and 1 metrics with constant $L_\nu$. This is a mild strengthening of the Wasserstein-1 condition in \cite{woobiadal23}, reflecting our second-moment analysis framework. No additional structural assumptions on $\uppi$ are required; the 
distribution $\uppi$ enters the analysis only through the  
second-moment and mean-square Lipschitz bounds.


\begin{figure*}
\vspace{0.5em}
		\centering
		\includegraphics[width = .95\textwidth]{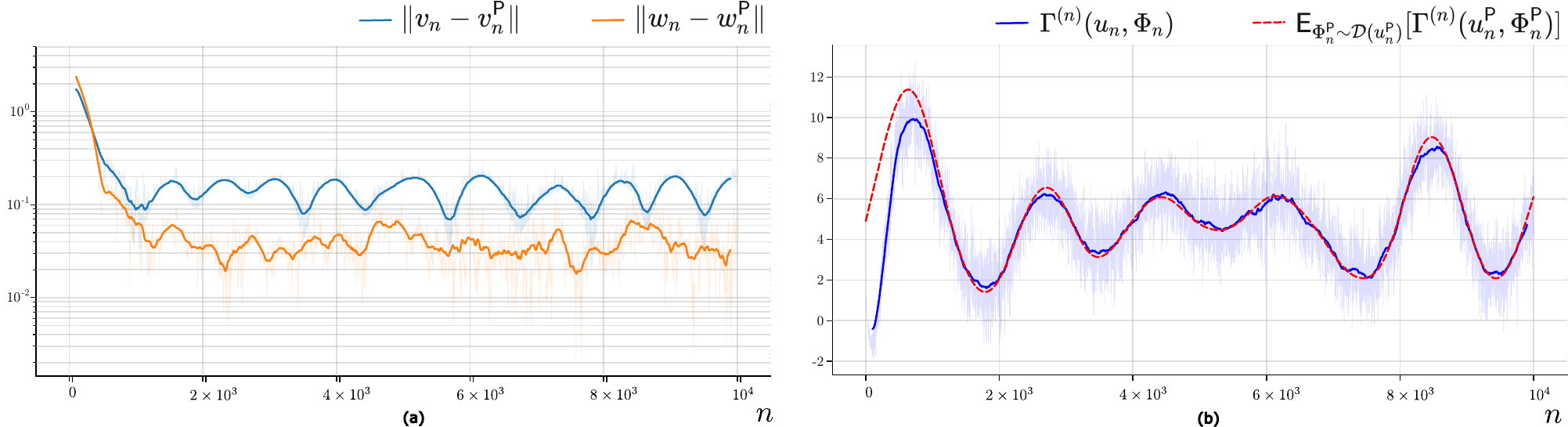}
		\caption{(a) Norm of the deviation between $\{u_n\}$ and $\{ \uPO_n \}$ as a function of $n$; (b) Evolution of $\{\Obj^{(n)} (u_n, \Phi_n) \}$. Transparent curves display 
instantaneous values, while opaque lines represent a moving 
average over a window of $200$ iterations.}
		\label{fig:Tracking_perf}
        \vspace{-.4cm}
	\end{figure*}

\subsection{Analysis of the Stochastic Primal-Dual Algorithm}

Before presenting the main results, we establish key properties of the operator $\nabla \Uppsi$ and discuss the relationship between optimal and performatively stable solutions.

It is argued in \cite[Sec.~3.1]{kosneduda11} that under (A2) 
and (A5), there exists a sequence of sets 
$\Lambda^{(n)} := \big\{ \lambda \in \Re^M_+ : 
\sum_{i=1}^M \lambda_i \leq b^{(n)}_\Lambda \big\}$
with $b^{(n)}_\Lambda > 0$, such that   
$ \sum_{i=1}^M\lambdaPO_{n,i} <    b^{(n)}_\Lambda$
for 
each $n$. The surrogate sets $\{\clH^{(n)}\}$ are defined in (A2)
with the same structure. Let $\barb_\Lambda \eqdef \sup_{n\ge1} b_\Lambda^{(n)}$ and $b^{\clU,\clH,\Lambda} \eqdef \max\{\barb_\clU,\barb_\clH,\barb_\Lambda \}$.

\begin{lemma}
\label[lemma]{t:facts}
Suppose that (A1)--(A3) hold along with (A5). Suppose moreover that $\barb_\Lambda<\infty$. Then,

\whamrm{(i)}$\nabla \Uppsi^{(n)}$ is Lipschitz continuous in quadratic mean: there is $L_\Uppsi<\infty$ such that for each $n$ and all $z,z' \in \clU^{(n)} \times (\clH^{(n)} \cup \Lambda^{(n)})$ that are $\clF_n$-measurable with $\Phi \sim \clD(u)$ and $\Phi' \sim \clD(u')$,
		\[
	\Expect[\|\nabla \Uppsi^{(n)}(z,\Phi)   -  \nabla \Uppsi^{(n)}(z',\Phi')   \|^2 \mid \clF_n]
		\leq 
		L_\Uppsi^2 (\| z - z' \|^2 + \| \Phi - \Phi' \|_{4,n}^2) \, .
		\]

\whamrm{(ii)} $\nabla \barUppsi^{(n)}$ is strongly monotone in its first variable: there is $\mu_\Uppsi>0$ such that for each $n$, all $z,z' \in \clU^{(n)} \times (\clH^{(n)}\cup \Lambda^{(n)})$, and all $u \in \clU^{(n)}$,
	 	\[
	 		[
	 	\nabla\barUppsi^{(n)}(z;u)
        - \nabla\barUppsi^{(n)}(z';u)  ]
	 	^\transpose 
	 	(z - z') 
	 	\geq 
	 	\mu_\Uppsi \| z -z'   \|^2 \, .
	 	\]

        \whamrm{(iii)} There exists a constant $\sigma_\Uppsi < \infty$ such that the following holds for each $n$:
    	 	\[
	 	\Expect[\|  \nabla \Uppsi^{(n)}(z_n,\Phi_{n}) -
         \Expect[ \nabla \Uppsi^{(n)}(z_n,\Phi_{n})  \mid \clF_n ]
        \|^2 \mid \clF_n]
	 	\leq 
	 	\sigma^2_\Uppsi \, .
	 	\]

\whamrm{(iv)}$\nabla \barUppsi^{(n)}$ is Lipschitz continuous in  its second variable: for each $n$, all  $z \in \clU^{(n)} \times (\clH^{(n)} \cup \Lambda^{(n)})$, and all $u,u' \in \clU^{(n)}$, 
\[
\|\nabla\barUppsi^{(n)}(z;u) - \nabla\barUppsi^{(n)}(z;u')\|^2 \leq L_\Uppsi^2 L_{\nu}^2 \| u -u'\|^2 \, ,
\]
        in which $L_\Uppsi$ is defined in part (i) and $L_\nu$ is the constant in \eqref{e:Phi_Lip}.
\end{lemma}

In particular, if $\Phi_n$ is a deterministic function of $u_n$, $ \nabla \Uppsi^{(n)}(z_n,\Phi_{n}) =
         \Expect[ \nabla \Uppsi^{(n)}(z_n,\Phi_{n})  \mid \clF_n ]$ for each $n$ and one may take $\sigma_\Uppsi=0$ in \Cref{t:facts}~(iii).

Existence of the optimal saddle points $\{z^*_n\}$ 
follows from standard minimax theory provided $\barclL^{(n)}(z;u)$ is convex-concave and $\clU^{(n)}$, $\Lambda^{(n)}$ are convex and compact for each $n$ \cite{woobiadal23}. Establishing existence of the 
performatively stable points $\{\zPO_n\}$ is more subtle, as 
it requires a fixed-point argument as shown in \cite{woobiadal23}.

The following is \cite[Prop. 2.11]{woobiadal23} and establishes existence and uniqueness of $\{\zPO_n\}$. 

\begin{lemma}
\label[lemma]{t:zP}
Suppose that (A1)--(A3) hold along with (A5). Suppose moreover that $\frac{1}{\mu_\Uppsi} L_\Uppsi L_\nu<1$ and let $\{z^*_n\}$ denote the sequence solving \eqref{e:minmax_prob}.  
Then, the sequence $\{\zPO_n\}$ exists, is unique, and satisfies:
$
\|   z^*_n - \zPO_n  \| \leq  \frac{2}{\mu_\Uppsi} L_\nu L_\Uppsi  \max\{ b^{(n)}_\clU , b^{(n)}_\Lambda   \}
$.
\hfill \qed
\end{lemma}

As discussed in \Cref{s:Intro}, full information about $\nabla \Uppsi$ or the sequence $\{\Lambda^{(n)}\}$ cannot be expected in the setting of this paper. Instead, we  implement
the approximate projected primal-dual scheme \eqref{e:PSGD_noreg_online} to estimate $\{\zPO_n\}$ in which
$
\nabla \haUppsi^{(n)} =
\begin{bmatrix}
\nabla_u \haclL^{(n)}
&
-\nabla_\lambda \haclL^{(n)}
\end{bmatrix}^\transpose
$
where $  \nabla_\lambda \haclL^{(n)}= 
 \begin{bmatrix}
     \haXi_1^{(n)}
    &
    \cdots 
    &
      \haXi_M^{(n)}
\end{bmatrix}^\transpose
  - \eta \lambda_n$ and
\[
 \nabla_u \haclL^{(n)} = F^\transpose \nabla g_y^{(n)}(\hat{y}_n) 
  + \nabla_u g_u^{(n)}(u_n, \Phi_{n}) 
  +  \sum_{j=1}^M \lambda_{n,j} \nabla_u \haXi_j^{(n)} + \mu u_n
\]
with 
$\haXi^{(n)} =  \begin{bmatrix}
     \haXi_1^{(n)}
    &
    \cdots 
    &
     \haXi_M^{(n)}
\end{bmatrix}^\transpose$, $ \haXi_i^{(n)}
=
{g^{i,(n)}_\Xi} (\hay_n)$ and $\nabla_u \haXi_i^{(n)}
=
 F^\transpose \nabla {g^{i,(n)}_\Xi} (\hay_n)$.

The following result is analogous to \cite[Lemma~4]{berdal19} and bounds 
the approximation error of $\nabla\haUppsi$ in terms of the 
measurement error $\varepsilon_m$ from (A4). 

Let $ b_y \eqdef  \|F\| \barb_\clU + \|G \| \sigma_\Delta + \| D\| \sigma_\Delta$.
\begin{lemma}
		\label[lemma]{t:estimation_haclL}
		Suppose that (A1)--(A5) hold. Then,
		$
		\|  \nabla \haUppsi^{(n)} -  \nabla \Uppsi^{(n)}(z_n,\Phi_{n})  \|_{2,n} 
		\leq 
		b_{\haUppsi} \epsy_m
		$,
        in which $b_{\haUppsi} = L_g\| F\|(1+M \barb_\clH) + \sqrt{M}[b^\bullet+ L_g (2 b_y +\epsy_m)]$.
	\end{lemma}

Since the surrogate dual sets $\{\clH^{(n)}\}$ may not 
contain the performatively stable dual variables, the 
projection onto $\clH^{(n)}$ may differ from the projection 
onto  $\Lambda^{(n)}$. This mismatch is quantified by the following result; notably, the error vanishes 
whenever $\Lambda^{(n)} \subseteq \clH^{(n)}$, which suggests choosing 
$b^{(n)}_\clH$ conservatively large in practice.
Let
    \[
    \clG^{(n)}(\clU,\Lambda,\zPO) \eqdef \Proj_{\clU^{(n)} \times \Lambda^{(n)}} \Big\{\zPO_{n}  - \alpha \begin{bmatrix}
		     \nabla_u \barclL^{(n)}(\zPO_n;\uPO_n) 
             \\
             - \nabla_\lambda\barclL^{(n)}(\zPO_n;\uPO_n)  
		 \end{bmatrix}\Big\}\,.
    \]

    \begin{corollary}
		\label[corollary]{t:diam_bdd}
        Under (A2) and the assumption that $\barb_\Lambda<\infty$, we  have that for each $n$,
        $
\|\clG^{(n)}(\clU,\Lambda,\zPO)  - \clG^{(n)}(\clU,\clH,\zPO) \| \leq \epsy_\clH 
        $,
in which $\epsy_\clH = \sup_{n} \max\{b^{(n)}_\Lambda - b^{(n)}_\clH ,0 \} $. 
\hfill  \qed
\end{corollary}

We now state the main result.

\begin{theorem}
		\label[theorem]{t:online_tracking_result_limsup}
		Suppose (A1)--(A5) hold. Suppose moreover that $\barb_\Lambda<\infty$ and that there exists $\barpsi<\infty$ such that $ \| \zPO_{n+1} - \zPO_n \|\leq \barpsi$ for all $n$ and $\frac{1}{\mu_\Uppsi}L_\Uppsi L_\nu<1$. Let $\{z_n\}$ be the iterates of~\eqref{e:PSGD_noreg_online}. Then, the bound \eqref{e:MSE_lim} holds for $\alpha < \frac{\mu^e_\Uppsi}{4 L_\Uppsi^2(1 + L_\nu^2)}$
    with:
    \[ \begin{aligned}
		\rho_a & \eqdef \frac{4}{\mu^e_\Uppsi}  \sigma^2_\Uppsi   \, , \quad  \rho_b  \eqdef \alpha \frac{4}{\mu^e_\Uppsi}    b^2_{\haUppsi} \epsy^2_m   + \frac{4 \sqrt{2}}{\mu^e_\Uppsi} b_{\haUppsi} \epsy_m b^{\clU,\clH,\Lambda} \, ,
		\\
         \rho_c &\eqdef \frac{1}{\mu^e_\Uppsi} \barpsi^2  \, ,
        \quad
        \rho_d  \eqdef \frac{1}{\mu^e_\Uppsi } \epsy_\clH(\epsy_\clH + 4 \sqrt{2} b^{\clU,\clH}  )
        \, , \quad   b_\circ \eqdef \frac{4}{\mu^e_\Uppsi},
		\end{aligned}
		\]
        and $\mu^e_\Uppsi = \mu_\Uppsi - L_\Uppsi L_\nu$.
	\end{theorem}

\vspace{.1cm}



In~\eqref{e:MSE_lim}, the step-size 
$\alpha$ balances stochastic averaging against tracking of a moving 
target: decreasing $\alpha$ reduces the stochastic contribution but 
amplifies the terms associated with time variation and surrogate-set 
mismatch, whereas increasing $\alpha$ has the opposite effect. Hence, the 
bound cannot generally be made arbitrarily small by tuning $\alpha$ 
alone.

We also stress one of the main contributions: in many engineering applications, the magnitude of the optimal dual variables is not known a priori. The sets $\{\mathcal H^{(n)}\}$ therefore act as implementable approximations of the unknown dual constraint sets. The resulting mismatch introduces an error captured by $\rho_d$. The regularization parameters $\mu$ and $\eta$ 
enter the analysis through the strong-monotonicity parameter 
$\mu_\Uppsi=\min\{\mu,\eta\}$ and hence affect the admissible step-size 
and the constants in~\eqref{e:MSE_lim}. Smaller values of $\mu$ and 
$\eta$ reduce the bias introduced by regularization, but they also weaken strong 
monotonicity.

\section{Numerical Experiment}
We illustrate our method on a problem  emulating the optimization of a power system with 
price-responsive assets.

\wham{Setup:} 
The system models a virtual power plant with heterogeneous assets, including three uncontrollable loads, three price-responsive loads, and three controllable photovoltaic (PV) units. The price-responsive loads can represent, for example, data centers providing grid services. The control input is 
$u = [v \; w]^\transpose$, 
where $v \in \Re^3$ denotes PV power injections and 
$w \in \Re^3$ denotes price incentives. Each consumer 
responds to the price signal via 
$\Phi_n = E w_n + \xi_n$, where $\{\xi_n\}$ is a sequence of i.i.d.\ random variables with mean 
$\mu_\xi$ and standard 
deviation $\sigma_\xi$, representing baseline 
consumption of flexible loads.
The scalar output is the 
feeder head power $y_n = h^{(n)}(u_n,\Phi_n) = \mathbf{1}^\transpose v_n + 
\mathbf{1}^\transpose \Phi_n + \mathbf{1}^\transpose r_n$, 
where $\{r_n\}$ collects uncontrollable power.    
The 
optimization problem is
\begin{align}
  &\min_{u \in \clU^{(n)}} 
  \Expect_{\Phi \sim \clD(u)}[\Obj^{(n)}(u,\Phi)]
  \nonumber \\
  &\text{subject to: } \quad 
  \Expect_{\Phi \sim \clD(u)}[h^{(n)}(u,\Phi)] = P_0^n
  \, , 
  \\
  &\phantom{\text{subject to: }} \quad \Expect_{\Phi \sim \clD(u)}[|h^{(n)}(u,\Phi) - P_0^n|^2] 
  \leq \epsy_P
  \nonumber
\end{align}
where
$\Obj^{(n)}(u,\Phi) = c_D \|v - \bar{P}_n\|^2 + 
c_P({w}^\transpose \Phi + 
\|M_{\text{reg}} w\|^2)$
balances tracking of available PV power $\{\bar{P}_n\}$ 
against the cost of price incentives and $\{ P^n_0 \}$ is a desired power profile. 

\wham{Results:} The primal-dual scheme was implemented for a time horizon of $N =10^4$. The parameters were chosen as $\alpha = 5 \times 10^{-3}$, $\mu = \eta = 0.02$, $\epsy_P = 0.3$, 
$c_D = c_P = 1$, $M_{\text{reg}} = 0.15 I$, $\barb_\clH =15$, and $\clU^{(n)} \equiv \clU = \{ u \in \Re^6 : u^\transpose u \leq 15\}$. The reference signals $\{\bar{P}_n\}$, 
$\{P_0^n\}$, and uncontrollable load  profiles $\{r_n\}$ were chosen as sinusoidal curves as follows:
\[
r_n = \begin{bmatrix}
    -0.3 + 0.3\sin(2 \pi n/2000)
    \\
        -0.5 + 0.4\sin(2 \pi n/3000)
    \\
        -0.7 + 0.5\sin(2 \pi n/4000)
\end{bmatrix}
\, , \qquad 
\barP_n = \begin{bmatrix}
    0.5 + 0.4\sin(2 \pi n/1500)
    \\
        0.6 + 0.5\sin(2 \pi n/2000)
    \\
        0.7 + 0.6\sin(2 \pi n/2500)
\end{bmatrix}
\]
and $\barP^n_0 = -6.0 + 0.8\sin(2 \pi n/1500)$.

The random variable $\{ \xi_n\}$ was chosen Gaussian with $\mu_\xi = [-1.0, -1.5, -2.0]^\transpose$, $\sigma_\xi = 0.3$ and $E = \text{diag}(0.4, 0.5, 0.6)$. Real-time measurements were emulated through the model $\hay_n = y_n + \epsy^\bullet_n$ in which $\epsy^\bullet_n \sim \text{Unif}[-1/2,1/2]$.

\Cref{fig:Tracking_perf} shows the tracking error 
 between estimates and performatively stable decisions, and the evolution of 
the objective as functions of $n$. As expected from 
\Cref{t:online_tracking_result_limsup}, the algorithm tracks 
the performatively stable saddle points with bounded error after a transient period. 

\section{Technical Proofs}
Throughout this section, we fix $\gamma^0 \sim \uppi$ independent of $\clF_n$ and let $\Phi = \nu(u,\gamma^0)$, $\Phi' = \nu(u',\gamma^0)$.
	For each $n$, we use the following short-hand notation for tracking errors: $\tilz_n \eqdef z_n - \zPO_n$,
	$
	\nabla \tilUppsi^{(n)}(z_n,\Phi_{n}) \eqdef  \nabla \Uppsi^{(n)}(z_n,\Phi_{n}) - \nabla \Uppsi^{(n)}(\zPO_n,\PhiPO_{n})
	$. 

We begin by establishing  Lipschitz continuity and strong monotonicity of the map $\nabla \Uppsi$:
\begin{lemma}
		\label[lemma]{t:facts_f1}
		Suppose that (A1)--(A3) hold.  Then, for each $n$ and all $u,u' \in \clU^{(n)}$ that are $\clF_n$-measurable,

        \whamrm{(i)} $\nabla_u \Obj^{(n)}$ is Lipschitz continuous in conditional mean: there is $L_\Obj<\infty$ such that
		\[
		\|\nabla_u \Obj^{(n)}(u,\Phi)   - \nabla_u \Obj^{(n)}(u',\Phi')   \|_{2,n}
		\leq 
		L_\Obj (\| u - u'   \|+ \| \Phi - \Phi'   \|_{4,n}) \, .
        \]

        \whamrm{(ii)} For each $1\leq i\leq M$,
$\nabla_u \Xi_i^{(n)}$ is Lipschitz continuous in conditional mean: there is $L_\Xi<\infty$ such that
		\[
\|\nabla_u \Xi_i^{(n)}(u,\Phi)   - \nabla_u \Xi_i^{(n)}(u',\Phi')   \|_{2,n}
		\leq 
		L_\Xi  (\| u - u'   \|+ \| \Phi - \Phi'   \|_{4,n}) \, .
		\]

            \whamrm{(iii)} 
$\Xi^{(n)}$ is Lipschitz continuous in conditional mean: there is $\barb_\Xi<\infty$ such that 
		\[
		\| \Xi^{(n)}(u,\Phi)   - \Xi^{(n)}(u',\Phi')  \|_{2,n}
		\leq 
		\barb_\Xi (\| u - u'   \|+ \| \Phi - \Phi'   \|_{4,n}) \, .
        \]
	
\end{lemma}

\begin{proof}
    Part (i). For each $n$, denote $y_n = h^{(n)}(u,\Phi)$ and $y'_n = h^{(n)}(u',\Phi')$.
	By the triangle inequality we have 
    \begin{equation}
    \| \nabla_u \Obj^{(n)}(u,\Phi) 
            -
            \nabla_u \Obj^{(n)}(u',\Phi') 
			\|_{2,n} 
            \leq  
            e^a_n +  e^b_n
\label{e:Lipsch_step}
    \end{equation}
    in which 
    \[
    \begin{aligned}
    e^a_n  &=  \|\nabla_u \Obj_y^{(n)}(u,\Phi) 
            -
            \nabla_u \Obj_y^{(n)}(u',\Phi') \|_{2,n} 
            \\
            e^b_n &= \|\nabla_u g_u^{(n)}(u,\Phi)  - \nabla_u g_u^{(n)}(u',\Phi') \|_{2,n} \,  .
            \end{aligned}
            \]
    Employing the chain rule we obtain,
			$e^a_n
			= 
			\|  F^\transpose [\nabla g^{(n)}_y(y_n) -  \nabla g^{(n)}_y(y_n') ]
            \|_{2,n}  \leq
			\| F^\transpose \| \| \nabla g^{(n)}_y(y_n) -  \nabla g^{(n)}_y(y_n')\|_{2,n}$.
    The Lipschitz continuity conditions in (A1) 
    and the identity in \eqref{e:Phi_Lip} give 
   \[
     \begin{aligned}
    \| \nabla g^{(n)}_y(y_n) -  \nabla g^{(n)}_y(y_n') \|_{2,n} 
    &\leq  L_g    (\| F\|  \| u-u'\|  + \|G\| \|\Phi - \Phi' \|_{2,n} ) \, , 
    \\
    \| \nabla_u g^{(n)}_u(u,\Phi) -  \nabla_u g^{(n)}_u(u',\Phi') \|_{2,n}  
    &\leq   L_g( \| u - u'\| + \|\Phi - \Phi' \|_{2,n} ) \, .
    \end{aligned}
    \]
    The final result then follows by substituting  the above bounds into \eqref{e:Lipsch_step} and applying Jensen's inequality with $L_\Obj = 2 L_g \max\{\|F \|^2 ,\|F\|\|G \|,1\}$.

Part (ii) follows from the same steps as the ones employed in part (i), resulting in $L_\Xi = \|F \|  L_g \max\{\|F \| ,\|G \| \}$.

Part (iii). By the assumed Lipschitz continuity of $\nabla g^{i,(n)}_{\Xi}$ in (A1), it follows that for any $y^\circ \in [y_n,y_n']$, $\| \nabla g^{i,(n)}_{\Xi} (y^\circ)\| \leq b^\bullet + L_g( \|y_n\| + \|y'_n\|) = e^c_n $ and an application of the mean value theorem then gives $| \Xi_i^{(n)}(u,\Phi) -  \Xi_i^{(n)}(u',\Phi')| \leq e^c_n \|y_n -y_n'\|$. The system model \eqref{e:model_lin}, $\clF_n$-measurability of $u,u'$, and the bounds $\|u\| \leq \barb_\clU$ and $\Expect_{\gamma^0 \sim \uppi}[ \|\nu(u,\gamma^0)\|^4 ] \leq \sigma^4_\Delta$ in (A2)--(A3) give $\| y_n \|_{4,n},\| y'_n \|_{4,n} \leq b_y$ with $b_y$ given before \Cref{t:estimation_haclL}, and $\| y_n - y'_n \|_{4,n} \leq \max\{\|F\|,\|G\| \}(\| u-u' \| + \| \Phi-\Phi' \|_{4,n})$. Thus, taking conditional $L_2$ norms of both sides, applying H{\"o}lder's inequality yields $(\Expect[| \Xi_i^{(n)}(u,\Phi) -  \Xi_i^{(n)}(u',\Phi')|^2 \mid \clF_n])^{1/2} \leq b_\Xi (\|u -u'\| + \| \Phi-\Phi' \|_{4,n})$ with $b_\Xi = \max\{\|F\|,\|G\| \} (b^\bullet + 2 L_g b_y)$. Collecting the $M$ components of $\Xi^{(n)}$ completes the proof with $\barb_\Xi = \sqrt{M} b_\Xi$.
\end{proof}





\smallskip
\whamit{Proof of parts (i) and (ii) of \Cref{t:facts}.}
We begin the proof with the recognition that under (A1), $\Obj^{(n)}$ and $\Xi^{(n)}$ are convex in their first variable.
In view of this and the results in \Cref{t:facts_f1}, the proof of parts (i) and (ii) follows from similar steps as the ones employed in \cite[Lemma 3.4]{kosneduda11}. In particular, $L_ \Uppsi = \sqrt{2} \sqrt{(L_\Obj + \max\{\barb_\clH, \barb_\Lambda\} \sqrt{M} L_\Xi+ \barb_\Xi + \mu )^2
+ 
( \barb_\Xi + \eta)^2}$ and $\mu_\Uppsi = \min\{\mu,\eta\}$.
\hfill
\qed



Next, we establish the remaining parts  of \Cref{t:facts}.
    \begin{lemma}
    \label[lemma]{t:facts_f3}
    Under (A1)--(A3), there exists a constant $\sigma_\clL$ such that the following holds for each $n$ and each $1 \leq j\leq M$:
    \whamrm{(i)}
	 	$ \displaystyle
	 	\Expect[\|  \nabla_u \Obj^{(n)}(u_n,\Phi_{n}) -
         \Expect[ \nabla_u \Obj^{(n)}(u_n,\Phi_{n}) \mid \clF_n]
        \|^2 \mid \clF_n]
	 	\leq 
	 	\sigma^2_\clL
	 	$,
        \whamrm{(ii)} 
        $ \displaystyle
	 	\Expect[\|  \lambda_{n,j} \nabla_u \Xi_j^{(n)}(u_n,\Phi_{n}) -
         \Expect[  \lambda_{n,j} \nabla_u \Xi_j^{(n)}(u_n,\Phi_{n})\mid \clF_n]
        \|^2 \mid \clF_n]
	 	\leq 
	 	\sigma^2_\clL
	 	$,
            \whamrm{(iii)}
            	$ \displaystyle
	 	\Expect[\|  \Xi^{(n)}(u_n,\Phi_{n}) -
         \Expect[ \Xi^{(n)}(u_n,\Phi_{n})\mid \clF_n]
        \|^2 \mid \clF_n]
	 	\leq 
	 	\sigma^2_\clL
	 	$.
    \end{lemma}
    \begin{proof}
	Let $\Phi^u_{n} = \nu(u_n,\gamma_n)$ and $y^u_n = h^{(n)}(u_n,\Phi^u_{n})$ with $\gamma_n \sim \uppi$ independent of $\clF_n$.
 	Using the triangle inequality and the  assumed Lipschitz continuity of $\nabla g^{(n)}_y$ in (A1), we have the following for a constant $b^\bullet$,

\[\begin{aligned}
		\|  \nabla_u \Obj^{(n)}(u_n,\Phi^u_{n}) \|_{2,n}
		&
		\leq 
		\| F^\transpose \nabla g^{(n)}_y(y^u_n) +\nabla_u g^{(n)}_u(u_n,\Phi^u_n) \|_{2,n}
		\\
        &\leq
        b^\bullet (
        \| F^\transpose\|+1)
        +
        L_g(\| F^\transpose\|
        \|  y^u_n\|_{2,n}
+\|  u_n\|_{2,n}
        +
         \|  \Phi^u_n\|_{2,n})
		\leq b^\dagger 
	\end{aligned}
	\]
   with $b^\dagger  = b^\bullet (
       \| F\|+1) + L_g ( \| F\|b_y + \barb_\clU + \sigma_\Delta)
$  where $b_y$ is given before \Cref{t:estimation_haclL}.
The last inequality is due to $u_n$ being $\clF_n$-measurable, the system model \eqref{e:model_lin}, the bounds $\|u_n\| \leq \barb_\clU$ and $\Expect_{\gamma_n \sim \uppi}[ \|\nu(u_n,\gamma_n)\|^4 ] \leq \sigma^4_\Delta$ in (A2)--(A3), and Jensen's inequality.

Part (ii) follows from similar steps as part (i) giving $\| \lambda_{n,j} \nabla_u \Xi_j^{(n)}(u_n,\Phi^u_{n}) \|_{2,n} 
\leq b^\triangle$ with $b^\triangle = \barb_\clH\| F\| (b^\bullet + L_g b_y)$.

For the functions in part (iii), the mean value theorem yields the bound $| g^{i,(n)}_{\Xi} (y^u_n)| \leq b^\bullet + \| \nabla g^{i,(n)}_{\Xi} (y^\circ)\| \|y^u_n\|$ for any $y^\circ \in [0, y^u_n]$, in which the Lipschitz conditions in (A1) give $\| \nabla g^{i,(n)}_{\Xi}(y^\circ)\| \leq b^\bullet + L_g\| y^u_n \|$. Taking conditional $L_2$ norms of both sides, applying H{\"o}lder's inequality, and collecting the components of $\Xi^{(n)}$ gives $\|  \Xi^{(n)}(u_n,\Phi^u_{n}) \|_{2,n} \leq b^\diamond$ with $b^\diamond = \sqrt{M}[b^\bullet(1 + b_y) + L_g b_y^2]$.


    Applications of the conditional Jensen's inequality and the triangle inequality complete the proof for $\sigma_\clL = 2 \max\{ b^\dagger , b^\triangle, b^\diamond  \}$. 
    \end{proof}

\smallskip
\whamit{Proof of parts (iii) and (iv) of \Cref{t:facts}.}
The triangle inequality gives 
\[\|   \nabla \Uppsi^{(n)}(z_n,\Phi_{n}) -
         \Expect[ \nabla \Uppsi^{(n)}(z_n,\Phi_{n})\mid \clF_n]
        \|_{2,n} \leq \epsy^a_n + \epsy^b_n + \epsy^c_n +\epsy^d_n \, ,
        \]
    where
         \[
 \begin{aligned}
\epsy^a_n   &=\|  \nabla_u \Obj^{(n)}(u_n,\Phi_{n}) -
         \Expect[ \nabla_u \Obj^{(n)}(u_n,\Phi_{n})\mid \clF_n]
        \|_{2,n} \, , 
\\
\epsy^b_n  &= \sum_{j=1}^M  \| \lambda_{n,j} \nabla_u \Xi_j^{(n)}(u_n,\Phi_{n}) -
         \Expect[ \lambda_{n,j} \nabla_u \Xi_j^{(n)}(u_n,\Phi_{n})\mid \clF_n]
        \|_{2,n}
\,,
\\
\epsy^c_n &=\|  \Xi^{(n)}(u_n,\Phi_{n}) -
         \Expect[ \Xi^{(n)}(u_n,\Phi_{n})\mid \clF_n]
        \|_{2,n}
\,,
\\
\epsy^d_n  &= \mu \| u_n -\Expect[u_n\mid \clF_n]  \|_{2,n}+  \eta \| \lambda_n -\Expect[\lambda_n\mid \clF_n]  \|_{2,n} 
\, .
\end{aligned}
\]
Since $  \lambda_n $ and $ u_n $ are $\clF_{n}$-measurable, we have that $\epsy^d_n=0$. 
The proof of part (iii) is complete upon applying  \Cref{t:facts_f3} which yields $\sigma_\Uppsi = (M+2)\sigma_\clL $.

Part (iv). The Lipschitz bound in part (i) yields that for all $z,z' \in \clU^{(n)} \times (\clH^{(n)} \cup \Lambda^{(n)} )$, $\|\nabla \Uppsi^{(n)}(z,\Phi)   -  \nabla \Uppsi^{(n)}(z,\Phi')   \|_{2,n}
		\leq 
		L_\Uppsi \| \Phi - \Phi' \|_{4,n}$. 
The proof is complete upon applying Jensen's inequality to the left hand side of the above equation along with the assumed Lipschitz bound \eqref{e:Phi_Lip}, and later squaring both sides.
\hfill \qed



		The following shorthand notation will be used throughout the remainder of this section: let $\clE_n = \nabla \haUppsi^{(n)} -   \nabla \Uppsi^{(n)}(z_{n},\Phi_{n})$ and $\clM_n  \eqdef \clM^a_n +  \clM^b_n+\clM^c_n $, in which $ \clM^a_n  \eqdef  \nabla \Uppsi^{(n)} (z_{n},\Phi_{n}) -   \Expect[\nabla \Uppsi^{(n)} (z_{n},\Phi_{n})  \mid \clF_n]$, $\clM^b_n  \eqdef     \Expect[\nabla \Uppsi^{(n)} (z_{n},\Phi_{n})  \mid \clF_n]-   \Expect[\nabla \Uppsi^{(n)} (z_{n},\PhiPO_{n})  \mid \clF_n]
        $, and   
	\[
		 \clM^c_n  
         \eqdef \Expect[\nabla \Uppsi^{(n)} (z_{n},\PhiPO_{n})  \mid \clF_n]  - \begin{bmatrix}
		     \nabla_u \barclL^{(n)}(\zPO_n;\uPO_n)
             \\
             - \nabla_\lambda\barclL^{(n)}(\zPO_n;\uPO_n)
		 \end{bmatrix} \, .
	\]
    
The next two lemmas bound the terms $\|\clE_n\|_{2,n}$  and
$\Expect[\clM_n^\transpose \tilz_n \mid \clF_{n}]$, which are crucial for  obtaining a contraction for the MSE of estimates.

\smallskip
\whamit{Proof of \Cref{t:estimation_haclL}.}
		From the triangle inequality, we have
		\[ 
		\|  \nabla \haUppsi^{(n)} -  \nabla \Uppsi^{(n)}(u_n,\Phi_{n})  \|_{2,n} 
		\leq 
		c^a_n + c^b_n + c^c_n 
		\]
		in which
		\[
		\begin{aligned}
			c^a_n  &= \| F^\transpose  [\nabla g_y^{(n)}(\hay_n) - \nabla g_y^{(n)}(y_n)] \|_{2,n} 
            			\\
            c^b_n&= \sum_{j=1}^M  \| \lambda_{n,j} [\nabla_u \haXi^{(n)}_j -\nabla_u\Xi_j^{(n)}(u_n,\Phi_{n})] \|_{2,n} 
			\\
            c^c_n&= \| \haXi^{(n)} - \Xi^{(n)}(u_n,\Phi_{n}) \|_{2,n} 
		\end{aligned}
		\]
		We proceed to bound each term. The assumed Lipschitz continuity of $\nabla g^{(n)}_y$ in (A1) and the bounds in \Cref{t:facts_f1}~(iii) yield the following: $c^a_n   \leq L_g  \| F \| \| \hay_n - y_n \|_{2,n}$.
        Similarly, the Cauchy-Schwarz inequality, (A3) and \Cref{t:facts_f1}~(ii) give
        \[ \begin{aligned}
			c^b_n   &\leq \| F \| \sum_{j=1}^M  \| \lambda_{n,j} [\nabla g^{j,(n)}_\Xi(\hay_n) -\nabla g^{j,(n)}_\Xi(y_n)] \|_{2,n} 
            \\
            & \leq  \| F \| \sum_{j=1}^M  | \lambda_{n,j} | L_g \| \hay_n - y_n \|_{2,n} 
         \leq  M  \barb_\clH L_g  \| F \|  \| \hay_n - y_n \|_{2,n} 
\end{aligned}
		\]

It remains to bound the last term. Using similar steps as in the proof of \Cref{t:facts_f1}~(iii), it follows that for any $y^\circ \in [y_n,\hay_n]$, $\| \nabla g^{i,(n)}_{\Xi} (y^\circ)\| \leq b^\bullet + L_g( \|y_n\| + \|\hay_n\|)$ where, again, $\|y_n\|_{4,n} \leq b_y$ and $\|\hay_n\|_{4,n} \leq  \|\hay_n - y_n\|_{4,n} + \|y_n\|_{4,n} \leq b_y + \epsy_m$. Applying the mean value theorem, taking conditional $L_2$ norms, and using H{\"o}lder's inequality, we obtain 
$
(\Expect[| \haXi_i^{(n)} - \Xi_i^{(n)}(u_n,\Phi_{n}) |^2 \mid \clF_n])^{1/2} \leq [b^\bullet+ L_g (2 b_y +\epsy_m)]  \| \hay_n - y_n\|_{4,n}
$.

The proof is complete upon collecting the $M$ components of $\Xi^{(n)}$.
\hfill
\qed

\begin{lemma}
		\label[lemma]{t:cross}
		Suppose that (A1)--(A5) hold and that $\barb_\Lambda<\infty$. Then, the following hold:
        \whamrm{(i)}
        $ \displaystyle
\Expect[{\clM^c_n} \mid \clF_{n}] =\Expect[\nabla \Uppsi^{(n)} (z_{n},\PhiPO_{n})\mid \clF_{n}] - \Expect[\nabla \Uppsi^{(n)} (\zPO_{n},\PhiPO_{n})\mid \clF_{n}] $,

                \whamrm{(ii)}
        $ \displaystyle
        \Expect[\clM_n^\transpose \tilz_n \mid \clF_{n}] \geq \mu^e_\Uppsi \| \tilz_n\|^2
        $, in which $\mu^e_\Uppsi=\mu_\Uppsi - L_\nu L_\Uppsi$.
	\end{lemma}
	\begin{proof}
    Part (i) follows from the dominated convergence theorem:
    the steps in the proof of \Cref{t:facts_f3} lead to the bound: $ \sup_{z \in \clZ^{(n)}}\| \nabla \Uppsi^{(n)}(z,\PhiPO_n)\| \leq b_0 + b_1 \| \PhiPO_n\| + b_2 \| \PhiPO_n\|^2$, in which $b_0,b_1,b_2$ are constants and  $\clZ^{(n)}$ is a compact set containing a neighborhood of  $\clU^{(n)} \times \Lambda^{(n)}$. Upon taking expectations and using the bounds in (A3), we have $\Expect_{\PhiPO_n \sim \clD(\uPO_n)}[ \sup_{z \in \clZ^{(n)}}\| \nabla \Uppsi^{(n)}(z,\PhiPO_n)\|] <\infty $ and hence the sequence is dominated and the order of expectation and differentiation can be exchanged in the second term of $\clM^c_n$. Moreover, the expectation coincides with the conditional expectation since $\zPO_n$ is deterministic and $\PhiPO_n \sim \clD(\uPO_n)$ is independent of $\clF_n$.
    
    The proof of part (ii) begins with the observation: $
\Expect[\clM_n^\transpose \tilz_n \mid \clF_{n}] =\Expect[(\clM^b_n+ \clM^c_n)^\transpose \mid \clF_{n}] \tilz_n  \geq \Expect[{\clM^c_n}^\transpose \mid \clF_{n}] \tilz_n - |\Expect[{\clM^b_n}^\transpose  \mid \clF_{n}]\tilz_n |
$,
which holds since $\tilz_n$ is $\clF_n$-measurable and $\Expect[\clM^a_n \mid \clF_n]=0$ by definition. The proof is complete in view of the result of part (i) and upon applying parts (ii) and (iv) of \Cref{t:facts} to bound the right hand side.
	\end{proof}

\begin{lemma}
		\label[lemma]{t:one_step_couple_haobj}
		Under  the assumptions of \Cref{t:online_tracking_result_limsup}, the following holds:
		\[
		\Expect[\|  z_{n+1} - \clG^{(n)}(\clU,\clH,\zPO) \|^2  \mid \clF_{n} ]
		\leq 
		\Upupsilon_\alpha \| \tilz_{n} \|^2
		+   \clQ_\alpha \, ,
		\]
         in which 
$
         \clQ_\alpha  = 4\alpha^2( b^2_{\haUppsi} \epsy^2_m + \sigma_\Uppsi^2) + 4 \sqrt{2} \alpha b_{\haUppsi} \epsy_m b^{\clU,\clH,\Lambda}
    $
      and 
      $\Upupsilon_\alpha = 1 - 2 \alpha \mu^e_\Uppsi + 4 \alpha^2 L_\Uppsi^2(1+L_\nu^2)$ with $\mu^e_\Uppsi =  \mu_\Uppsi - L_\Uppsi L_\nu$.
	\end{lemma}
\begin{proof}
Let \[\beta_{n} =\zPO_{n}  - \alpha \begin{bmatrix}
		     \nabla_u \barclL^{(n)}(\zPO_n;\uPO_n)  
             \\
             - \nabla_\lambda\barclL^{(n)}(\zPO_n;\uPO_n)  
		 \end{bmatrix}
         \, .
         \]
 Using the non-expansiveness property of the projection operator, we have 

		\[
		\begin{aligned}
			 \|  z_{n+1} - \clG^{(n)}(\clU,\clH,\zPO)  \|^2
			&\leq 
			\| z_{n}  - \alpha \nabla \haUppsi^{(n)}  -\beta_{n} \|^2
			\\
			& =
			\| \tilz_{n} \|^2+ \alpha^2 \|\clE_{n}  + \clM_{n} \|^2 -2\alpha [  {\clE_{n}}  + {\clM_{n}}]^\transpose \tilz_{n}
            \\
            & \leq
			\| \tilz_{n}\|^2+ \alpha^2 \|\clE_{n}  + \clM_{n} \|^2 -2\alpha {\clM_{n}}^\transpose \tilz_{n}
            +
            2\alpha| {\clE_{n}}^\transpose \tilz_{n}|
		\end{aligned}
		\]

		Upon taking conditional expectations of both sides, the middle term is upper bounded as follows:
        \[
\|\clE_{n} \! + \! \clM_{n} \|^2_{2,{n}} \! \! \leq \!4 ( \|\clE_{n} \|^2_{2,{n}} \!\!+\! \| \clM^a_{n} \|^2_{2,{n}} \!+\! \| \clM^b_{n} \|^2_{2,{n}} \!+\! \| \clM^c_{n} \|^2_{2,{n}}\!)
 \! \leq \! 
 4(b^2_{\haUppsi} \epsy^2_m + \sigma_\Uppsi^2) + 4 L_\Uppsi^2(1+L_\nu^2)\|\tilz_{n}\|^2 \, ,
 \]
        in which the last bound is obtained using  \Cref{t:estimation_haclL}, \Cref{t:cross}~(i), parts (i), (iii) and (iv) of \Cref{t:facts}, and the fact that $\| u_{n} - \uPO_n\| \leq \| \tilz_{n}\|$ for each $n$.
        
		It remains to bound the last two terms. The Cauchy-Schwarz inequality gives $|\Expect[\clE_{n}^\transpose \tilz_{n} 
        \mid \clF_{n}]| \leq \| \clE_{n}\|_{2,{n}} \|\tilz_n\| \leq 2 \sqrt{2} b_{\haUppsi} \epsy_m b^{\clU,\clH,\Lambda}$ where the last bound follows from \Cref{t:estimation_haclL} and the assumed bound in (A2). Applying \Cref{t:cross}~(ii) to bound $\Expect[\clM_{n}^\transpose \tilz_{n} 
        \mid \clF_{n}]$ completes the proof. 
	\end{proof}

    The surrogate set mismatch enters through the following 
refinement of \Cref{t:one_step_couple_haobj}.
	
	\begin{lemma}
		\label[lemma]{t:one_step_clip}
		Under  the assumptions of \Cref{t:online_tracking_result_limsup}, the following holds:
		\[
		\Expect[\|  z_{n+1} - \zPO_{n} \|^2  \mid \clF_{n} ]
		\leq 
		\Upupsilon_\alpha \| \tilz_{n} \|^2
		+ \clB_\alpha \, ,
\]
    in which $\clB_\alpha = \clQ_\alpha + \epsy_\clH^2 +  4 \sqrt{2}\epsy_\clH b^{\clU,\clH}$ with 
    $\clQ_\alpha$ as defined by \Cref{t:one_step_couple_haobj} and $\epsy_\clH$ as given in \Cref{t:diam_bdd}.
    	\end{lemma}
	\begin{proof}
       The proof begins with the recognition that $\zPO_{n}$ solves the following fixed-point equation:
       $
       \zPO_{n} =  \clG^{(n)}(\clU,\Lambda,\zPO)
       $.
        Then, we have 
            $
       \|z_{n+1} - \zPO_{n}\|^2  =  
       \| \epsy^a_{n} \|^2  - 2(\epsy^a_{n})^\transpose \epsy^b_{n} + \|\epsy^b_{n} \|^2
       $,
        in which $\epsy^a_{n}  =  z_{n+1} - \clG^{(n)}(\clU,\clH,\zPO)$ and $ \epsy^b_{n} =   \clG^{(n)}(\clU,\Lambda,\zPO)  - \clG^{(n)}(\clU,\clH,\zPO) $.

Upon taking conditional expectations of both sides of the above equation, we bound the first  and last terms via \Cref{t:one_step_couple_haobj} and \Cref{t:diam_bdd}, respectively. To bound the remaining term, we apply the Cauchy-Schwarz inequality to obtain: $|\Expect[(\epsy^a_{n})^\transpose \epsy^b_{n}| \clF_{n}]| \leq \| \epsy^a_{n} \|_{2,{n}} \| \epsy^b_{n} \|_{2,{n}}$. 
The proof is complete after invoking \Cref{t:diam_bdd} 
and noting that, by (A2), 
$\| \epsy^a_{n}\| \leq 2 \sqrt{2} b^{\clU,\clH}$.
	\end{proof}


Let $\clR_n \eqdef (\zPO_{n-1} - \zPO_n)^\transpose 
(z_n - \zPO_{n-1})$. Next, we bound this 
cross-term, which arises from the time-variability of 
$\{\zPO_n\}$.

	\begin{lemma}
		\label[lemma]{t:middle_track}
		Under the assumptions of \Cref{t:online_tracking_result_limsup}, the following bound holds 
        for $n\geq 1$
        :
		\[
		|\Expect[\clR_n] | 
		\leq \barpsi \Big[
		\Upupsilon_\alpha^{n/2} \| \tilz_0 \|_2
            +
            \sqrt{\clB_\alpha} \frac{1 - \Upupsilon_\alpha^{n/2}}{1-\sqrt{\Upupsilon_\alpha}} + \barpsi \sqrt{\Upupsilon_\alpha}\frac{1 - \Upupsilon_\alpha^{(n-1)/2}}{1-\sqrt{\Upupsilon_\alpha}}
            \Big]
		\]
		where $\clB_\alpha $ is given by \Cref{t:one_step_clip}.
	\end{lemma}
\begin{proof}
We begin with the recognition that the Tower property of conditional expectations along with Jensen's inequality yield the upper bound: $|\Expect[\clR_n] | \leq \Expect[|\Expect[\clR_n \mid \clF_{n-1}]|]$.
 Then, we apply the Cauchy-Schwarz inequality to obtain 

		\begin{equation*}
			\begin{aligned}
			 |\Expect[\clR_n \mid \clF_{n-1}]| 
				&\leq
				\barpsi \| z_n - \zPO_{n-1}\|_{2,n-1}
				\\
				& \leq 
				\barpsi [\sqrt{\Upupsilon_\alpha} \| \tilz_{n-1} \|_{2,n-1}   + \sqrt{\clB_\alpha}]
                \\
				&  \leq 
				\barpsi [\sqrt{\Upupsilon_\alpha} \| z_{n-1} - \zPO_{n-2}\|_{2,n-1} + \sqrt{\Upupsilon_\alpha}\barpsi  +
            \sqrt{\clB_\alpha}]
			\end{aligned}
		\end{equation*}
        
		where the second and third bounds were obtained via \Cref{t:one_step_clip} and the triangle inequality, respectively.
		Repeating this process recursively and applying the finite geometric series formula yields the desired result.
	\end{proof}

\whamit{Proof of \Cref{t:online_tracking_result_limsup}.}
	Upon choosing $\alpha < \frac{\mu^e_\Uppsi}{4 L_\Uppsi^2(1 + L_\nu^2)}$, it follows that $\Upupsilon_\alpha \leq  \sqrt{\Upupsilon_\alpha} <1$ and $\Upupsilon_\alpha < 1-\mu^e_\Uppsi\alpha$. 

Expanding the square in $\|\tilz_n -\zPO_{n-1} + \zPO_{n-1}\|^2_{2,n-1}$ yields
	\[
	\begin{aligned}
		\| \tilz_n  \|^2_{2,n-1}
		&
        \leq
		\barpsi^2 + 
		\| z_n - \zPO_{n-1} \|^2_{2,n-1} + 2 \Expect[\clR_n \mid \clF_{n-1}]
		\\
		&
        \leq 
		\barpsi^2
		+
		\Upupsilon_\alpha \| \tilz_{n-1} \|^2 +  \clB_\alpha + 2 \Expect[\clR_n \mid \clF_{n-1}] 
	\end{aligned}
	\]
	where the last bound follows from  \Cref{t:one_step_clip}.

 Repeating this process recursively, taking expectations of both sides and then taking the limit supremum yields
\[
	\begin{aligned}
		\limsup_{N \to \infty} 
		\| \tilz_{N}  \|_2^2
		&\leq  \frac{ [ \clB_\alpha + \barpsi^2 ]}{\mu^e_\Uppsi \alpha} \!+ \! 2 \limsup_{N \to \infty} \sum_{i=0}^{N-1} \Upupsilon^i_\alpha |\Expect[\clR_{N-i}]| \, .
	\end{aligned}		
	\]

To bound the last term, we invoke \Cref{t:middle_track} and apply the geometric series formula once more to obtain

\[
\begin{aligned}
 \limsup_{N \to \infty} \sum_{i=0}^{N-1} \Upupsilon^i_\alpha |\Expect[\clR_{N-i}]|
  &\leq 
   \frac{(\barpsi \sqrt{\clB_\alpha} + \barpsi^2 \sqrt{\Upupsilon_\alpha})}{1-\sqrt{\Upupsilon_\alpha}}\frac{1}{\mu^e_\Uppsi\alpha}
    & \leq 
    (\barpsi \sqrt{\clB_\alpha} + \barpsi^2) \frac{2}{[\mu^e_\Uppsi]^2\alpha^2}
\end{aligned}
\]
in which the above bound was obtained by employing the fact that $\sqrt{\Upupsilon_\alpha}<1$. 
The proof is complete upon substituting the definition of $\clB_\alpha$ from \Cref{t:one_step_clip}.
	\hfill
	\qed

\section{Conclusions}
This paper studied OFO for constrained stochastic problems with decision-dependent distributions. A projected primal–dual algorithm with surrogate dual sets was analyzed, and mean-square tracking guarantees were established that quantify the effects of stochasticity, measurement noise, time variability, and dual-set mismatch.

\clearpage

%


\end{document}